\documentclass[12pt,reqno]{amsart}
\usepackage{amsfonts,latexsym,amssymb,amsmath,amsthm}
\usepackage{graphicx,hyperref,cite}

\newtheorem{proposition}{\bf{Proposition}}[section]
\newtheorem{theorem}{\bf{Theorem}}[section]
\newtheorem{remark}{\sc{Remark} }[section]

\begin{document}

\title[Sufficient conditions to thermoelectrochemical problem]{Sufficient conditions to
the existence for solutions of a thermoelectrochemical problem}
\author{Luisa Consiglieri}
\address{Luisa Consiglieri, Independent Researcher Professor,  European Union}
\urladdr{\href{http://sites.google.com/site/luisaconsiglieri}{http://sites.google.com/site/luisaconsiglieri}}

\begin{abstract} 
A mathematical model  is introduced for thermoelectrochemical
phenomena in an electrolysis cell,
and its
qualitative analysis is focused on existence of solutions.
 The model consists of a system of nonlinear
parabolic PDEs in conservation form expressing conservation of energy, mass
and charge. On the other hand, an integral form of Newton's law is used to describe
heat exchange at the electrolyte/electrode interface, a nonlinear radiation
condition is enforced on the heat  flux at the wall and a nonlinear boundary
condition is considered for the electrochemical  flux in order to account for
Butler-Volmer kinetics.
The main objective is the nonconstant character of each parameter,
 that is, the coefficients are assumed to be  dependent on the spatial variable
and the temperature. 
Making recourse of known estimates of solutions for some auxiliary elliptic and
parabolic problems, which are explicitly determined by
the Gehring-Giaquinta-Modica theory,
 we find sufficient smallness conditions on the data to guarantee
the existence of the original solutions
via the Schauder fixed point argument.
These conditions may provide useful informations for numerical as well as
real applications. We conclude with an example of application, namely the
electrolysis of molten sodium chloride.
\end{abstract}


\maketitle

\maketitle

\section{Introduction}

The conservative laws are universal in the description of the 
physicochemical phenomena.
Their particular applications depend on the transport coefficients behavior.
The introduction of the thermal effects into physicochemical
devices are being addressed by applied mathematicians \cite{mauri}. 
Quantitative description  of the heat rate data
is discussed in \cite{bedk,hansen}. 
The model parameters 
(such as the electrical mobilities $u_i$, and
 the thermal conductivity $k$,  among others) are 
assumed to be constant positive quantities whose values are specified
to  numerical simulations.
Our first shortcoming is that these coefficients are commonly discontinuous.

In view of the above discussion,
 we develop a thermoelectrochemical model for an electrolyte domain.
Our second shortcoming is that the physicochemical 
phenomena truly pass on the boundary of the domain.
We mention to \cite{solo} a mathematical modeling of the
 interaction of electric, thermal, and diffusion processes
in infinitely diluted solutions of electrolytes.
 The production of nuclear grade
heavy water, including water electrolysis, distillation, and chemical exchange
processes, provide a process matched to the feed supply
\cite{levins,ryan}. We refer to \cite{latzz}
 a mathematical model  of Li-ion batteries
 based exclusively on universally accepted principles 
of nonequilibrium thermodynamics 
and
 the assumption of the one step intercalation reaction at the interface of
electrolyte and active particles; and
 to \cite{licht,si}
other  attractive thermoelectrochemical approaches.

In thermoelectrochemical modeling, the force-flux
relations are  (see, at the steady-state, \cite{epjp} and the
references therein)
\begin{align}\label{defq}
{\bf q}=&-\mathsf{K}\nabla\theta-R\theta^2 \sum_{i=1}^ID'_i\nabla c_i
-\Pi\sigma\nabla\phi;\\
{\bf J}_i=&-c_iS_i\nabla\theta- D_{i}\nabla c_i
-u_ic_i\nabla\phi;\quad (i=1,\cdots,I)\nonumber \\
{\bf j}=&-\alpha\sigma\nabla\theta-F\sum_{i=1}^Iz_{i}D_{i}\nabla c_i
-\sigma\nabla\phi. \nonumber
\end{align}
Here, ${\bf q}$,  ${\bf J}_i$
 and $\bf j$ are, respectively, the measurable heat flux (in W$\cdot$m$^{-2}$), 
 the ionic flux of component $i$ (in mol$\cdot$m$^{-2}\cdot$s$^{-1}$),
  and  the electric current density  (in C$\cdot$m$^{-2}\cdot$s$^{-1}$). 
The unknown functions are  the temperature  $\theta$,
 the molar concentration vector ${\bf c}=(c_1,\cdots, c_I)$,
and  the electric potential $\phi$.
 Hereafter the subscript $i$ stands for the correspondence to the
ionic component $i$ 
intervener in the reaction process. 
As the problem involves several symbols, we summarize their notation
in Appendix. In particular, 
$\mathsf{K}$ denotes the thermal conductivity tensor, reflecting 
anisotropic properties of the medium.
Also the Peltier coefficient
$\Pi$ can be a tensor
\cite{becker}. By this reason, we keep
both $\alpha$ and $\Pi$ as known functions,
 although  the first Kelvin relation correlates 
$\Pi$ with the Seebeck coefficient 
 $\alpha$. All transport
coefficients can be either
experimentally measured or calculated  as dependent on temperature
and spatial variable, while 
the Soret effect and the
related Dufour effect include the concentration of the correspondent
ionic component \cite{huja,dios}.

Dealing with these issues, our main concerns are: in the physical
point of view to introduce thermal radiation
on one part of the boundary, to approach the Butler-Volmer equation 
on other part of the boundary; and in the mathematical
point of view to find sufficient explicit conditions on the data
to the existence of solutions, under minimal assumptions on
the transport coefficients, as consequence of the fixed point theory.
The key of an integrability exponent larger than $n$
 for the solution (say in $n$ space dimensions)
is the need of making severe restrictions on the corresponding 
leading coefficient
function - as is carried out in the literature \cite{evans}.
 
\section{Statement of the problem and main theorem}
\label{spmt}

Let $T>0$ be an arbitrary (but preassigned) time, and $\Omega$ represent
an electrolysis cell, which consists (as
in general) of two electrodes and an electrolyte. We abbreviate
$Q_T=\Omega \times ]0,T[$.

From the conservation of  energy,  the mass balance equations, and
 the conservation of electric charge, we derive, respectively, in $Q_T$
\begin{align}\label{teq}
\rho c_\mathrm{p}\frac{\partial \theta}{\partial t}+\nabla\cdot{\bf q}=0;\\
\frac{\partial c_i}{\partial t}+\nabla\cdot{\bf J}_i=0;\\
\nabla\cdot{\bf j}=0,\label{peq}
\end{align}
 where the density $\rho$ and the
specific heat capacity $c_\mathrm{p}$ (at constant volume) are assumed  to be (positive)  constants.
The absence of external forces, assumed in (\ref{teq})-(\ref{peq}), is due to their occurrence at the surface of the electrodes.

The boundary $\partial\Omega$ is decomposed into four pairwise
 disjoint open subsets $\Gamma_l$, $l=$ a, c, w, o, representing the anode,
 the cathode, the wall, and the (remaining) outer, 
 respectively, surfaces such that  (cf. Fig. 1) 
\[
\partial\Omega=\overline\Gamma_{\rm a}\cup\overline\Gamma_{\rm c}
\cup\overline\Gamma_{\rm w}\cup\overline\Gamma_{\rm o}.
\] 
For the sake of simplicity, we call the
electrode/electrolyte interface  $\Gamma_{\rm e}=
\Gamma_{\rm a}\cup \Gamma_{\rm c}$ by simply $\Gamma$, and
we set $\Sigma_T=\Gamma_{\mathrm{w}}\times ]0,T[$.
Hence further, 
for each $l=$ a, c, w,
 $\theta_l$ represents a given temperature at $\Gamma_l$, and
 $\bf n$ is the outward unit normal to the boundary $\partial\Omega$.
\begin{figure}\label{fig1}
\includegraphics[scale=0.7]{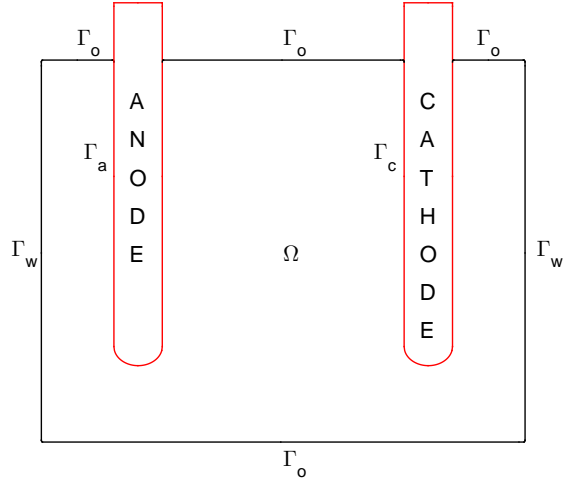}
\caption{Schematic
saggital representation of an electrolytic cell (with corners being smoothed by circumferences)}
\end{figure}

The parabolic-elliptic system (\ref{teq})-(\ref{peq})
is accomplished by the following boundary conditions.
For a.e. in $]0,T[$,
we consider the 
heat balance described by the global Newton law of cooling
\begin{equation}\label{heat}
\int_{\Gamma_\mathrm{a}}
{\bf q}\cdot{\bf n}\mathrm{ds}+
\int_{\Gamma_\mathrm{c}}{\bf q}\cdot{\bf n}\mathrm{ds}
=
\int_{\Gamma_\mathrm{e}}h_\mathrm{C}(\theta-\theta_\mathrm{e})\mathrm{ds},
\qquad \theta_\mathrm{e}
=\left\lbrace\begin{array}{ll}
\theta_\mathrm{a}&\mbox{on }\Gamma_\mathrm{a}\\
\theta_\mathrm{c}&\mbox{on }\Gamma_\mathrm{c}
\end{array}\right.
,
\end{equation}
where 
$h_{\rm C}$ denotes the conductive heat transfer coefficient.
By the constitutive law (\ref{defq}) of $\bf q$, 
 the left-hand side of (\ref{heat}) says that
the heat generated  is divided into the irreversible
reaction heat due to efficiency losses
 of the electrode reaction, and 
 the reversible  reaction heat mainly due to
the entropy change of the electrode reaction
which is called Peltier heat
and changes sign with changing current direction  (cf. \cite{gu}).

A gas bubble behavior at a
hydrogen-evolution electrode  was reported by some researchers
\cite{bore,kikuchi,ross}. This hydrogen gas generated at
the cathode causes turbulence of water or wastewater flow \cite{chen-chen}.
At each electrode/electrolyte interface ($l=$ a, c), we consider
\[
 -Fz_i{\bf J}_i\cdot{\bf n}_l=g_{i,l} (\cdot,\theta,\phi).
\]
Here,
$g_{i,l}$ may represent  the generalized Butler-Volmer kinetics 
that is composed by
the involved charge and mass balances 
in the charge-transfer reaction under illumination \cite{nie}, and
 the Butler-Volmer expression itself
\begin{equation}\label{bvl}
J_{l}\left(
\exp\left[\frac{\beta_i s_lF\eta}{R\theta}\right]-
\exp\left[-\frac{(1-\beta_i) s_lF\eta}{R\theta}\right]\right),
\end{equation}
where  $J_l$ represents the
transfer (exchange) current density due to the electrode reaction,
 $s_l$ is the stoichiometric
coefficient of electrons  in the anode/cathode  ($l=$ a, c),
 $\beta_i$ is the transfer coefficient  ($i=1,\cdots,I$),
and $\eta=\phi-\phi_{\rm eq}$ denotes  the surface
overpotential. 

Although the electroneutrality assumption says that 
${\bf j}=\sum_{i=1}^Iz_iF{\bf J}_i$, we consider on
$\Gamma\times ]0,T[$
\begin{equation}
-{\bf j}\cdot{\bf n}=g, \label{eass}\end{equation}
with $g$ being a prescribed surface electric current assumed to be
tangent to the surface for all $t>0$.
We refer as an open problem
the nonlocal Dirichlet boundary condition for the electric potential,
$\phi=j(I)$ \cite{ffh},
on the part of the boundary ($\Gamma_\mathrm{e}$) where the device is
connected to the circuit, 
with $j$ being  a nonlinear function and $I=\int_{\Gamma_\mathrm{e}}\sigma(\theta) \frac{\partial\phi}{ \partial n}$
denoting the total current, 
when the voltage drop across the electrical circuits is not prescribed
but is coupled to the remainder circuit.

Let temperature fulfill the radiative condition over
 $\Gamma_\mathrm{w}\times ]0,T[$
\begin{equation}
{\bf q}\cdot{\bf n}=h_\mathrm{R}
| \theta | ^{\ell-2}\theta -\gamma.\label{radia}
\end{equation}
This general exponent $\ell\geq 2$ \cite{lap}
accounts for the radiation behavior of the heavy
water electrolysis \cite{flei,kyr},
namely the  Stefan-Boltzmann radiation law if $\ell=5$
with 
$h_{\rm R}$  denoting the radiative heat transfer coefficient, {\em i.e.}
$h_\mathrm{R}=\sigma_{\rm SB}\epsilon$,
and $\gamma=\sigma_{\rm SB}\alpha \theta_\mathrm{w}^4$.
The parameters, the emissivity $\epsilon$ and the absorptivity $\alpha$, both
 depend on the spatial variable
and the temperature function $\theta$.

The following no outflows are considered:
\begin{align}
\mathrm{on }\ \Gamma_{\rm o}\times ]0,T[,&\qquad
\mathbf{q}\cdot{\bf n}=0;\\
\mathrm{on }\ \left(\Gamma_{\rm w}\cup\Gamma_{\rm o}\right)
\times ]0,T[,&\qquad
{\bf J}_i\cdot{\bf n}={\bf j}\cdot{\bf n}=0,\qquad (i=1,\cdots,I).\label{ji0}
\end{align}

Finally, the following initial conditions for all $x$ in $\Omega$ are assumed:
 \begin{equation}\label{ci}
 \theta(x,0)=\theta_0(x),\quad c_i(x,0)=c_i^0(x),\quad
i=1,\cdots,I.
 \end{equation}

In the framework of Sobolev and Lebesgue functional spaces, we 
use the following spaces of test functions:
\begin{align*}
V_{p,\ell}(Q_T)=&\{v\in L^p(0,T;W^{1,p}(\Omega)):\ 
v|_{\Sigma_T}\in
L^{\ell}(\Sigma_T)\};\\
V_p(\Omega)=&\{ v\in W^{1,p}(\Omega):\ \int_{\partial\Omega} v
\mathrm{ds}=0\},
\end{align*}
with their usual norms,  $p,\ell>1$.

In order to derive our variational problem,
 we note that every ionic mobility $u_i=z_{i}D_iF/(R\theta)$ 
 satisfies the Nernst-Einstein relation  $\sigma_i=F{z_i}{}u_ic_i$, with $\sigma_i=t_i\sigma$ representing
ionic conductivity, and $t_i$ is the transference 
number (or transport number) of species $i$.

Then our variational
problem under study is:

\noindent ($\mathcal P$)
 Find the triple
temperature--concentration--potential $(\theta,{\bf c},
\phi)$ such that verifies the variational problem:
\begin{align}
\rho c_\mathrm{p}\int_0^T\langle \partial_t \theta, v\rangle\mathrm{dt}+
\int_{Q_T}(\mathsf{K}(\cdot,\theta)\nabla\theta)\cdot \nabla v\mathrm{dx}
\mathrm{dt} 
+\int_{\Sigma_T}h_\mathrm{R} (\cdot,\theta)|\theta|^{\ell-2}
\theta v\mathrm{ds}\mathrm{dt} + \nonumber \\
+\int_0^T\int_{\Gamma}h_\mathrm{C} (\cdot,\theta)\theta v
\mathrm{ds}\mathrm{dt}= 
\int_0^T\int_{\Gamma}h_\mathrm{C}(\cdot,\theta)\theta_\mathrm{e}v
\mathrm{ds}\mathrm{dt}
+\int_{\Sigma_T}
\gamma(\cdot,\theta) 
v\mathrm{ds}\mathrm{dt} \nonumber\\
 -\int_{Q_T} \left(R\theta^2\sum_{i=1}^ID_i'(\cdot,c_i,\theta)\nabla c_i+
\Pi(\cdot,\theta) \sigma(\cdot,\theta)\nabla\phi\right) 
\cdot\nabla v\mathrm{dx}\mathrm{dt}
,\nonumber\\ \forall v\in V_{p',\ell}(Q_T); \label{pbt}
\\
\int_0^T\langle \partial_t c_i, v\rangle\mathrm{dt}+
\int_{Q_T}D_i(\cdot,\theta)\nabla c_i\cdot \nabla v\mathrm{dx}
\mathrm{dt} =
\int_0^T\int_{\Gamma}g_i(\cdot,\theta,\phi)v\mathrm{ds}\mathrm{dt}\nonumber\\
-
\int_{Q_T}\left(c_iS_i(\cdot,c_i,\theta)\nabla\theta+ 
\frac{t_i}{Fz_i}\sigma(\cdot,\theta)\nabla\phi
\right)\cdot \nabla v\mathrm{dx}
\mathrm{dt},\nonumber\\
 \forall v\in L^{p'}(0,T;W^{1,p'}(\Omega)),\quad i=1,\cdots,I
; \label{pbci}
\\
\int_\Omega\sigma(\cdot,\theta)\nabla\phi\cdot \nabla v\mathrm{dx}=
\int_{\Gamma}g v\mathrm{ds}\nonumber\\
-\int_\Omega
\left(\alpha(\cdot,\theta) \sigma(\cdot,\theta) \nabla\theta+
F\sum_{i=1}^Iz_iD_i(\cdot,\theta)\nabla c_i\right)
\cdot \nabla v\mathrm{dx},\nonumber \\
\forall v\in V_{p'}(\Omega), \qquad \mbox{a.e. in }[0,T[,\label{pbfi}
\end{align}
 where 
 $p'$ accounts for the conjugate exponent of $p$: $p'=p/(p-1)$.

We assume 
\begin{description}
\item[(H1)]  The  electrical conductivity, Peltier, Seebeck, Soret, Dufour,
and diffusion
 coefficients $\sigma,\Pi,\alpha,S_i,D_i',D_i$ ($i=1,\cdots,I$) are
  Carath\'eodory functions, 
{\em i.e.} measurable with respect to $x\in\Omega$ and
  continuous with respect to other variables, such that
   \begin{align}
   \exists\sigma_\#,\sigma^\#>0:\qquad& \sigma_\#\leq \sigma(x,e)\leq 
\sigma^\#; \label{smm}\\
\exists\Pi^\#>0:\qquad&
|\Pi(x,e){\bf a}|\leq \Pi^\# |{\bf a}|;\label{pmax}\\
\exists\alpha^\#>0:\qquad&
|\alpha(x,e)|\leq \alpha^\#;\label{amax}\\
\exists S_i^\#>0:\qquad& |dS_i(x,d,e)|\leq S_i^\#;\label{dsmax}\\
\exists (D_i')^\#>0:\qquad&  Re^2|D_i'(x,d,e)|\leq (D_i')^\#;\label{edmax}\\
\exists D_i^\#>0:\qquad& F|z_i|D_i(x,e)\leq D_i^\#;\label{dmax}\\
\exists (D_i)_\#>0:\qquad& D_i(x,e)\geq (D_i)_\#,\label{dmin}
  \end{align}
  a.e. $ x\in \Omega$, 
  for all ${\bf a}\in\mathbb{R}^n$,
and for all $d,e\in \mathbb R$.
  
\item[(H2)] The thermal conductivity
 $\mathsf{K}:\Omega\times\mathbb{R}\rightarrow\mathbb{M}_{n\times n}$ is a
  Carath\'eodory tensor, where $\mathbb{M}_{n\times n}$ denotes the set
  of $n\times n$ matrices, such that
 \begin{equation}
\exists k_\#>0:\quad K_{jl}(x,e)\xi_j \xi_l \geq k_\#|\xi|^2,
\quad\mbox{a.e. } x\in \Omega,\quad\forall e\in \mathbb{R},
\label{kmin}
  \end{equation}
  for all $\xi\in\mathbb{R}^n$,
under the summation convention over repeated indices:
$\mathsf{A}{\bf a}\cdot{\bf b}=A_{jl}a_j b_l ={\bf b}^\top  \mathsf{A}{\bf a}$;
and
 \begin{equation}
\exists k^\#>0:\quad | K_{jl}(x,e)|\leq k^\#,
\quad\mbox{a.e. } x\in \Omega,\quad\forall e\in \mathbb R,
\label{kmax}
  \end{equation}
  for all $j,l\in\{1,\cdots,n\}$.
  
 \item[(H3)]  The boundary operator $h_\mathrm{R}$ is a
 Carath\'eodory  function from $\Gamma_{\mathrm{w}}\times\mathbb{R}$
 into $\mathbb{R}$ such that
 \begin{equation}
\exists b_\#, b^\#>0:\quad b_\#\leq h_\mathrm{R}(x,e)\leq 
b^\#
\quad\mbox{a.e. } x\in \Gamma_{\mathrm{w}},\quad\forall e\in \mathbb R
.\label{bmm}
  \end{equation}
\item[(H4)]  The  transference  coefficient $t_i\in L^\infty (\Omega)$
is such that
   \begin{equation}
\exists t_i^\#>0:\quad
0\leq t_i(x)\leq F|z_i|t_i^\#, \quad\mbox{a.e. } x\in \Omega.\label{tmm}
\end{equation}

\item[(H5)]  For  some $\delta>0$,
$g\in L^{2+\delta}(\Gamma)$ such that $\int_{\Gamma}g\mathrm{ds}=0$.

\item[(H6)]  For  some $\delta>0$,
  $\theta_{\mathrm{e}}\in L^{2+\delta}(\Gamma\times ]0,T[)$, and
the boundary operators  $\gamma$ and
 $h_\mathrm{C}$ are
 Carath\'eodory  functions from $\Gamma_{\mathrm{w}}\times ]0,T[
\times\mathbb{R}$
 and $\Gamma\times ]0,T[\times\mathbb{R}$, respectively,
 into $\mathbb{R}$,
{\em i.e.} measurable with respect to $(x,t)$ and
  continuous with respect to the real variable. Moreover, they satisfy
 \begin{align}
\exists \gamma_{\mathrm{w}}\in L^{2+\delta}(\Sigma_T):
&\qquad |\gamma (x,t,e)|\leq 
\gamma_{\mathrm{w}}(x,t),
\quad\mbox{a.e. } x\in \Gamma_{\mathrm{w}};
\label{gmax}\\
\exists h_\mathrm{C}^\#>0:
&\qquad 0\leq h_\mathrm{C}(x,t,e)\leq h_\mathrm{C}^\#
,\quad \mbox{a.e. } x\in \Gamma,
\label{hmax}
  \end{align}
  a.e. $t\in ]0,T[$, and
for all $e\in \mathbb R$.
 
\item[(H7)]  For  some $\delta>0$, and 
for each $i=1,\cdots, I$, the boundary operator $g_i=g_{i,\mathrm{a}}\chi_{\Gamma_\mathrm{a}}+
g_{i,\mathrm{c}}\chi_{\Gamma_\mathrm{c}}$ is a
 Carath\'eodory  function from $\Gamma\times ]0,T[
\times\mathbb{R}\times\mathbb{R}$
 into $\mathbb{R}$ and there exists 
$\gamma_i\in L^{2+\delta}(\Gamma\times ]0,T[)$ such that
\begin{equation}
\exists g_i^\#\geq 0:\quad
|g_i(x,t,e,d)|\leq \gamma_i(x,t)+ g_i^\#(|d|+|e|),
\label{gimax}\end{equation}
  a.e. $ (x,t)\in \Gamma\times ]0,T[$, and for all $e,d\in \mathbb R$.
  
\item[(H8)]  For  some $\delta>0$, 
$\theta_0,c_i^0\in L^{2+\delta}(\Omega)$, $i=1,\cdots,I$.
\end{description}

 For the sake of simplicity, we assume in (H5)-(H8)
the same designation $\delta >0$.
Note that (\ref{gimax}) is verified for
a truncated version of the Butler-Volmer expression (\ref{bvl}).

The main interest of the mathematical model under study (governing
equations and boundary conditions) is strictly related to real world
applications (thermoelectrochemical phenomena in an electrolysis cell $\Omega$).
 In this respect, the consideration of a number $n$ of space dimensions
greater than $3$ is not really relevant. From the mathematical point of view, the broader
dimensional range, if available, is more meaningful in fact. Therefore, we state our main result
in the unified way.
\begin{theorem}\label{main}
Under the hypothesis (H1)-(H8), there exists a solution 
\[(\theta,
{\bf c},\phi)\in V_{p,\ell}(Q_T)\times
[L^p(0,T;W^{1,p}(\Omega))]^I\times
V_{p}(\Omega),
\]
for some $p>2$,
 to
(\ref{pbt})-(\ref{pbfi}) with the initial condition
(\ref{ci}) if provided by the smallness conditions  (\ref{datar}), and
(\ref{datar1})-(\ref{datari}).
\end{theorem}
\begin{remark}
The existence of $p$ is restricted to $[2,2+\delta]$, where $\delta>0$ is chosen smaller than
min$\lbrace 2/[n(\upsilon-1)],1/(\varkappa-1)\rbrace $ with $\upsilon,\varkappa>1$ being well-determined
constants by
the Gehring-Giaquinta-Modica theory \cite{ark94,ark95,gia83}.
\end{remark}

 \section{Strategy of the proof of Theorem \ref{main}}
 \label{sst}

In this section we discuss the key of the proof,
 and we recall a known result for the solvability.
The proof of Theorem \ref{main}  is based on the Schauder fixed point theorem
\cite{zi}.
We freeze the concentrations-temperature pair $({\bf c},\theta)$ in the
closed convex set
\begin{align*}
 \mathcal{K}=\lbrace ({\bf v},v)\in [L^p(0,T;W^{1,p}(\Omega))]^I\times
V_{p,\ell}(Q_T)
:\  \|\nabla v\|_{p,Q_T}+\| v\|_{\ell,\Sigma_T}\leq R,\\
\|\nabla v_i\|_{p,Q_T}+\| v_i\|_{p,Q_T}\leq R_i,\ i=1,\cdots,I
\rbrace,
\end{align*}
where $p,\ell\geq 2$,  and we built the well defined functional $\mathcal T$ 
such that
\begin{equation}\label{deft}
({\bf c},\theta )\in\mathcal{K}
\mapsto \phi\in V_{p}(\Omega)\mapsto
({\bf\Psi},\Theta),
\end{equation}
where $\phi$, $\bf \Psi$, and
 $\Theta$ are the unique functions given at Propositions \ref{lemaep},
\ref{lemaci}, and \ref{lemata}, respectively. 
Their proofs rely on  existence results due 
to a weak reverse H\"older inequality for local solutions 
\cite{ark94,ark95,gia83}.
For reader's convenience, we recall the parabolic existence result \cite{ark95,gia83}.
\begin{theorem}\label{tmain}
Let $\Omega$ be a $C^{1}$ domain,  $T>0$, and the assumptions
 (\ref{kmin})-(\ref{bmm}) be fulfilled.
There exists $\upsilon>1$ such that for any
$0<\delta<2/[n(\upsilon-1)]$ and  $p\in [2,2+\delta]$
if  ${\bf f}\in {\bf L}^{2+\delta}(Q_T)$, 
$f\in L^{2+\delta}(\Gamma\times ]0,T[)$, 
$H\in L^{2+\delta}(\Sigma_T)$ and
$u_0\in L^{2+\delta}(\Omega)$,
then the variational problem 
\begin{align}
\int_0^T\langle \partial_t u, v\rangle\mathrm{dt}+\int_{Q_T}
( \mathsf{K}\nabla u)\cdot \nabla v\mathrm{dx}
\mathrm{dt}+\int_{\Sigma_T} h_{\mathrm{R}}(u)|u|^{\ell-2}
uv\mathrm{ds}\mathrm{dt}=\nonumber\\
=\int_{Q_T} {\bf f}\cdot \nabla v\mathrm{dx}
\mathrm{dt}+
\int_0^T\int_{\Gamma}fv \mathrm{ds}\mathrm{dt}+
\int_{\Sigma_T} Hv\mathrm{ds}\mathrm{dt},\quad\forall
v\in V_{p',\ell}(Q_T),\label{wvf}
\end{align}
has a  solution $u$ in $L^{p,\infty}(Q_T)\cap V_{p,\ell+p-2}(Q_T)$ such that
 $\partial_t u\in  [V_{p',\ell}(Q_T)]'$, and it verifies
 \begin{align}
{\rm ess} \sup_{t\in [0,T]}\| u\|_{p,\Omega}^p(t) \leq
\mathcal{H}(k_\#,b_\#,p)
\exp\left[(p-1)T\right];\label{gr1}\\
\|u\|^{\ell+p-2}_{\ell+p-2,\Sigma_T}\leq 
(b_\#)^{-1} 
\mathcal{H}(k_\#,b_\#,p)\left(1+(p-1)T\exp\left[(
p-1)T\right]\right);\\
\|\nabla u\|_{p,Q_T}\leq \mathcal{C}(k_\#)^{-1}
\left[\sqrt{ k_\#\mathcal{H}(k_\#,b_\#,2)\left(1+T\exp\left[T\right]\right)}+
\right. \nonumber\\
+\left. 
\sqrt{1+k_\#}\left( \|{\bf f}\|_{p,Q_T}
 +K_{2n/(n+1)}\left[\|f\|_{p,\Gamma\times ]0,T[}
 +\|H\|_{p,\Sigma_T} \right]\right)\right]
,\label{cota2}
\end{align}
with
\begin{align*}
\mathcal{H}(k_\#,b_\#,p)= \| u_0\|_{p,\Omega}^p+
 \left({p-1\over k_\#}\right)^{p/2}
\|{\bf f}\|_{p,Q_T}^p+ \nonumber\\
+\frac{p(\ell-1)}{(\ell+p-2)
 b_\#^{(p-1)/(\ell-1)}}\int_{\Sigma_T}
|H|^{\frac{\ell+p-2}{\ell-1}} \mathrm{ds}\mathrm{dt}+\\
+(p-1)\left( \left(\frac{p^2}{2k_\#(p-1)}\right)^{1/(p-1)}
+1\right)K_{2n/(n+1)}^{2/(p-1)}|\Omega|^{[(p-1)n]^{-1}}
 \|f\|_{p',\Gamma\times ]0,T[}^{p'}.
\end{align*}
Here,  
 $K_{2n/(n+1)}$ stands for  the continuity constant of the trace embedding
$W^{1,2n/(n+1)}(\Omega)\hookrightarrow L^2(\Gamma)$, and
$\mathcal{C}$ is a positive constant depending only on $\upsilon$, $p$,
$n$, and $\Omega$.
In particular, 
if  $ b_\#=0$ and $f=0$,
 then (\ref{gr1}) and (\ref{cota2}) remain true by 
replacing $\mathcal{H}(k_\#,b_\#,p)$ by
\begin{align}
\mathcal{H}(k_\#,p)= \| u_0\|_{p,\Omega}^p+
\left({p-1\over k_\#}\right)^{p/2}
\|{\bf f}\|_{p,Q_T}^p+ \nonumber\\
+(p-1)\left( \left(\frac{p^2}{2k_\#(p-1)}\right)^{\frac{1}{p-1}}
+1\right)K_{2n/(n+1)}^{2/(p-1)}|\Omega|^{[(p-1)n]^{-1}}
 \|H\|_{p',\Sigma_T}^{p'}.\label{defhp}
\end{align}
\end{theorem}
\begin{remark}\label{ral}
By the Aubin-Lions theorem \cite{l}, we have that $u\in L^p(Q_T)$,
and the initial condition $u(0)=u_0$ makes sense at least in $L^p(\Omega)$.
\end{remark}

\section{Existence of auxiliary solutions}
\label{sec:2}

Let us establish the existence of solutions according to Section
\ref{sst}. Fix $\delta\in ]0,2/[n(\upsilon-1)][$ with  $\upsilon>1$ being
given from Theorem \ref{tmain}.

First, let us recall the existence of 
the  required auxiliary potential solving a second order elliptic equation
 of divergence form with a discontinuous leading coefficient.
 \begin{proposition}[Auxiliary potential]\label{lemaep}
 Let $\delta>0$, $t\in ]0,T[$,
$\theta(t), c_i(t)\in W^{1,2+\delta}(\Omega)$, 
 for every  $i=1,\cdots,I$,   
$g\in L^{2+\delta}(\partial\Omega)$ verify
 $\int_{\partial\Omega}g\mathrm{ds}=0$, 
and (\ref{smm}), (\ref{amax}), and (\ref{dmax}) hold. 
There exists $\varkappa>1$ such that  the Neumann problem (\ref{pbfi})
is uniquely (up to constants) solvable in $W^{1,p}(\Omega)$
for any $p\in [2,2+\delta]\cap [2,2+1/(\varkappa-1) [$.
Moreover, for each $]0,T[$ we have
\begin{align}
\sigma_\#\| \nabla \phi\|_{2, \Omega}\leq
K\|g\|_{2,\Gamma}+ 
\sigma^\#\alpha^\# \|\nabla\theta \|_{2, \Omega}
+\sum_{j=1}^I  D_{j}^\# \|\nabla c_j\|_{2, \Omega}; \ \label{cotaf2}\\
\| \nabla \phi\|_{p,\Omega}
\leq 
M_1\| \nabla \phi\|_{2, \Omega} 
+M_2(\sigma_\#)^{-1} 
\sqrt{ 1+\sigma_\# } \| {\mathcal F }(\theta,{\bf c})\|_{p, \Omega}
+\nonumber\\
+M_3 (\sigma_\#)^{-1}\sqrt{2+{ 2^{-1/n}\sigma_\#}}
 \| g\|_{p,\Gamma}
,\quad\label{cotaep}
\end{align}
where $K$ stands for a  positive constant depending on $n$ and $\Omega$,
 \[
{\mathcal F}(\theta,{\bf c}) = 
\sigma^\#\alpha^\# |\nabla\theta |
+\sum_{j=1}^I  D_{j}^\# |\nabla c_j|,
\] 
and  $M_1$, $M_2$ and $M_3$ are positive constants depending on $n$, $p$, $\varkappa$, and $\Omega$.
\end{proposition}
\begin{proof}
The existence of the weak unique solution satisfying (\ref{cotaf2}) is classical
 (for details see, for instance, \cite{lc-ijpde}).
A similar proof of the regularity estimate (\ref{cotaep}) can be found in \cite{ark94,ark95}.
\end{proof}

The existence of the
auxiliary concentrations-temperature pair
$({\bf \Psi},\Theta)$ is consequence of Theorem \ref{tmain} as follows.
 \begin{proposition}[Auxiliary concentrations]\label{lemaci}
 Let $\theta\in L^p(0,T; W^{1,p}(\Omega))$,
and $\phi \in V_{p}(\Omega)$ be in accordance with Proposition
\ref{lemaep},    with $p\in [ 2,2+\delta] \cap [2,2+1/(\varkappa-1)[$. 
Under the assumptions (\ref{smm}), (\ref{dsmax}), (\ref{dmax})-(\ref{dmin}),
 (\ref{tmm}), and (\ref{gimax}),
 there exists a  function ${\bf \Psi}\in  [L^p(0,T;W^{1,p}(\Omega))]^I$ 
being the unique 
 solution  to the variational problem,
 for each  $ i=1,\cdots,I$, 
 \begin{align}
\int_0^T\langle \partial_t \Psi_i, v\rangle\mathrm{dt}+
\int_{Q_T}D_i(\theta)\nabla \Psi_i\cdot \nabla v\mathrm{dx}
\mathrm{dt} =
\int_0^T\int_{\Gamma}g_i(\theta,\phi)v\mathrm{ds}\mathrm{dt}\nonumber\\
-
\int_{Q_T}\left( c_iS_i(c_i,\theta)\nabla\theta+ 
\frac{t_i}{Fz_i}\sigma(\theta)\nabla\phi
\right)\cdot \nabla v\mathrm{dx}
\mathrm{dt},\label{pbcia}
\end{align} 
for all  $v\in L^{p'}(0,T;W^{1,p'}(\Omega))$.
In particular,  $\partial_t {\bf \Psi}\in [L^p(0,T; [W^{1,p'}(\Omega)]')]^I$,
and ${\bf \Psi}\in [C([0,T]; L^2(\Omega))]^I$.
  Moreover,   for every  $ i=1,\cdots,I$, we have
\begin{align}
\|\Psi_i\|_{p,Q_T}^p \leq T
\|\Psi_i\|_{\infty,p,Q_T}^p \leq T 
\exp\left[(p-1)T\right]\left[ \|c_{0,i}\|_{p,\Omega}^p+
\right.\nonumber\\
+\left({p-1\over (D_i)_\#}\right)^{p/2}\left(
S_i^\#\|\nabla\theta\|_{p,Q_T}+
t_i^\#\sigma^\#\|\nabla\phi\|_{p,Q_T}\right)^p+ \nonumber\\
+\left( \left(\frac{p^2(p-1)^{p-2}}{2(D_i)_\#}\right)^{ \frac{1}{p-1}}
+p-1\right) K_{2n/(n+1)}^{2p'/p}|\Omega|^{p'(pn)^{-1}}
\left(
 \|\gamma_i\|_{p',\Gamma\times]0,T[}+\right.
\nonumber\\
\left.\left. +g_i^\# K_{pn/( n+p-1)}
|\Omega|^{1-1/p}
\left(\|\nabla\theta\|_{p,Q_T}+\|\theta\|_{p,Q_T}
+P_p\|\nabla\phi\|_{p,Q_T}\right)\right)^{p'}\right] ;\quad \label{cotapsiip}
\end{align}\begin{align}
\|\nabla \Psi_i\|_{p,Q_T} \leq \mathcal{C} (D_i)_\#^{-1}
\left[
 \sqrt{(D_i)_\#(1+T\exp[T])}\|c_{0,i}\|_{2,\Omega} +
\mathcal{G}_i^\#+\right. \nonumber\\
+ X_i\|\nabla\phi\|_{2,Q_T}+Y_i
\|\nabla\phi\|_{p,Q_T}+ \nonumber\\
\left. +\left(S_i^\#\mathcal{Z}(|Q_T|^{1/2-1/p},(D_i)_\#,1)
+\mathcal{Q}_i\right)(\|\nabla\theta\|_{p,Q_T}+\|\theta\|_{p,Q_T})\right],\quad
 \label{cotapsii}
 \end{align}
 with
\begin{align}
\mathcal{G}_i^\#= 
K_{2n/( n+1)}\left(\sqrt{(1+T\exp[T])(2+(D_i)_\#)
|\Omega|^{1/n}}\|\gamma_i\|_{2,\Gamma\times]0,T[}+\right.\nonumber \\ 
\left. +\sqrt{1+(D_i)_\#}
\|\gamma_i\|_{p,\Gamma\times]0,T[}\right); \label{gii}\\
X_i=\sqrt{1+T\exp[T]} \left( t_i^\#\sigma^\#
+g_i^\#\sqrt{2+(D_i)_\#}
|\Omega|^{\frac{1+1/n}2}K_{2n/( n+1)}^2P_2\right);\\
Y_i= \sqrt{1+(D_i)_\#}
\left( t_i^\#\sigma^\#
+g_i^\#K_{2n/( n+1)} K_{pn/( n+p-1)}|\Omega|^{1-1/p}
P_p \right);\\\mathcal{Q}_i= K_{2n/( n+1)}g_i^\#\left(
\sqrt{1+(D_i)_\#}K_{pn/(n+p-1)}|\Omega|^{1-1/p}
+\right.\nonumber \\ \left.
+\sqrt{(1+T\exp[T])(2+(D_i)_\#)
|\Omega|^{1/n}}K_{2n/(n+1)}|\Omega|^{1-1/p}T^{1/2-1/p}\right);\\
\mathcal{Z}(a,d,e)=  
a\sqrt{1+T\exp[T]} +e \sqrt{1+d}
,\quad a,d,e>0, \label{defz}
\end{align}
and $P_p$ stands for the Poincar\'e constant correspondent to the space
exponent $p$.
\end{proposition}
\begin{proof}
The existence of the required auxiliary concentrations is consequence of
Theorem \ref{tmain} and Remark \ref{ral}. In particular, we have
\begin{align*}
\|\nabla \Psi_i\|_{p,Q_T} \leq \mathcal{C} (D_i)_\#^{-1}
 \left[ \sqrt{(D_i)_\#(1+T\exp[T])}\|c_{0,i}\|_{2,\Omega} +
H_i(\theta,\phi)+ \right.\nonumber\\
 +t_i^\#\sigma^\#\left(
\sqrt{1+T\exp[T]}
\|\nabla\phi\|_{2,Q_T}+ \sqrt{1+(D_i)_\#}
\|\nabla\phi\|_{p,Q_T}\right) \nonumber\\
\left. +S_i^\#\left(
\sqrt{1+T\exp[T]}
\|\nabla\theta\|_{2,Q_T}+ \sqrt{1+(D_i)_\#}
\|\nabla\theta\|_{p,Q_T}\right) \right],
 \end{align*}
with 
\begin{align*}
H_i(\theta,\phi)=\mathcal{G}_i^\#+\\
+ K_{2n/( n+1)}g_i^\#\left(
 \sqrt{1+(D_i)_\#}\left(\|\theta\|_{p,\Gamma\times]0,T[}+
\|\phi\|_{p,\Gamma\times]0,T[}\right)
+\right. \\ \left.
+\sqrt{(1+T\exp[T])(2+(D_i)_\#)
|\Omega|^{1/n}}\left(\|\theta\|_{2,\Gamma\times]0,T[}+
\|\phi\|_{2,\Gamma\times]0,T[}\right)\right).
\end{align*}
Then, (\ref{cotapsii}) holds by taking the following inequalities into account
\begin{align*}
\|v\|_{p,\Gamma}\leq K_{pn/(n+p-1)}|\Omega|^{1-1/p}
 \left(\|\nabla v\|_{p,\Omega}+\|v
\|_{p,\Omega}\right);\\
\|w\|_{p,\Gamma}\leq  K_{pn/(n+p-1)}|\Omega|^{1-1/p}
P_p \|\nabla w\|_{p,\Omega},
\end{align*}
for all $v\in W^{1,p}(\Omega)$ and $w\in V_p(\Omega)$.

With analogous argument, we find  (\ref{cotapsiip}).
\end{proof}

 \begin{proposition}[Auxiliary temperature]\label{lemata}
 Let $\theta,c_i\in L^p(0,T;W^{1,p}(\Omega))$, $i=1,\cdots,I$, 
$\phi\in V_{p}(\Omega)$ be 
in accordance with Proposition
\ref{lemaep}, where $p\in[ 2,2+\delta]\cap [2,2+1/(\varkappa-1)[$, and
 the assumptions (\ref{smm}), (\ref{pmax}), (\ref{edmax}), 
 (\ref{kmin})-(\ref{bmm}), and (\ref{gmax})-(\ref{hmax}) be fulfilled.
Then, the variational problem
\begin{align}
\rho c_\mathrm{p}\int_0^T\langle \partial_t \Theta, v\rangle\mathrm{dt}+
\int_{Q_T}(\mathsf{K}(\theta)\nabla\Theta)\cdot \nabla v\mathrm{dx}
\mathrm{dt}
+\int_{\Sigma_T}h_\mathrm{R} (\theta)|\Theta|^{\ell-2}
\Theta v\mathrm{ds}\mathrm{dt} \nonumber \\
+\int_0^T\int_{\Gamma}h_\mathrm{C} (\theta)\Theta v
\mathrm{ds}\mathrm{dt}=
\int_0^T\int_{\Gamma}h_\mathrm{C}(\theta)\theta_\mathrm{e}v
\mathrm{ds}\mathrm{dt}
+\int_{\Sigma_T}
\gamma(\theta) 
v\mathrm{ds}\mathrm{dt} \nonumber\\
 -\int_{Q_T} \left(R\theta^2\sum_{j=1}^ID_j'(c_j,\theta)\nabla c_j+
\sigma(\theta) \Pi (\theta)\nabla\phi\right) 
\cdot\nabla v\mathrm{dx}\mathrm{dt},\qquad \label{pbta}
 \end{align}
 for all $v\in V_{p',\ell}(Q_T),$
is uniquely  solvable in $V_{p,\ell}(Q_T)$.
In particular,  $\partial_t {\Theta}\in L^p(0,T; [W^{1,p'}(\Omega)]')$, and
$ {\Theta}\in C([0,T]; L^2(\Omega))$.
  Moreover, the following estimates hold:
\begin{align}
\|\Theta\| _{\infty,p,Q_T}\leq 
\mathcal{H}_0^{1/p}(\|\nabla\phi\|_{p,Q_T},\|\nabla \mathbf{c}\|_{p,Q_T}) \exp[(p-1)T/p]
; \qquad \label{cotatinfp}
\\
\|\Theta\|^{\ell+p-2}_{\ell+p-2,\Sigma_T}\leq 
\frac{1+(p-1)T \exp[(p-1)T]}{(\rho c_\mathrm{p})^{-1} b_\#}
\mathcal{H}_0(\|\nabla\phi\|_{p,Q_T},\|\nabla \mathbf{c}\|_{p,Q_T}) 
;\label{cotatlp}
\\
\|\nabla \Theta\|_{p,\Omega} \leq \mathcal{C} (k_\#)^{-1}
 \left[ \sqrt{ \rho c_\mathrm{p} k_\#(1+T\exp[T])}
\|\theta_{0}\|_{2,\Omega}+\mathcal{H}^\#\right.
+\nonumber\\ +
\sigma^\#\Pi^\#\mathcal{Z}( \|\nabla\phi\|_{2,Q_T}, (\rho c_\mathrm{p})^{-1} k_\#,
\|\nabla\phi\|_{p,Q_T}) +\nonumber \\
\left.
+\mathcal{Z}(|Q_T|^{1/2-1/p}, (\rho c_\mathrm{p})^{-1} k_\#,  \sum_{j=1}^I (D_j')^\#
\|\nabla c_j\|_{p,Q_T} ) \right],\label{cotata}
 \end{align}
with $
\gamma_{\mathrm{e}}:= h_\mathrm{C}^\# |\theta_{\mathrm{e}}|,$  $\mathcal{Z}$ is given as (\ref{defz}), and
\begin{align*}
\mathcal{H}_0(a,\mathbf{b})=\|\theta_0\|_{p,\Omega}^p+(\rho c_\mathrm{p})^{-p/2}
\left(\frac{p-1}{k_\#}\right)^{p/2}
\left(
\sigma^\#\Pi^\#a+\sum_{j=1}^I (D_j')^\# b_j \right) ^p+ \\
+(\rho c_\mathrm{p})^{-1} \frac{p(\ell-1)}{(\ell+p-2)
 b_\#^{(p-1)/(\ell-1)}}\int_{\Sigma_T}
|\gamma_{\mathrm{w}}|^{\frac{\ell+p-2}{\ell-1}} \mathrm{ds}\mathrm{dt}+\\ 
+(\rho c_\mathrm{p})^{-p'} \left( \left(\frac{p^2(p-1)^{p-2}}{2k_\#(\rho c_\mathrm{p})^{-1} }\right)^{1\over p-1}
+p-1\right)K_{2n/(n+1)}^{2/(p-1)}|\Omega|^{\frac{1}{(p-1)n}}
 \|\gamma_{\mathrm{e}}\|_{p',\Gamma\times ]0,T[}^{p'}  ;
\end{align*}\begin{align*}
\mathcal{H}^\#=
 \sqrt{1+(\rho c_\mathrm{p})^{-1} k_\#}
K_{2n/(n+1)}(\|\gamma_{\mathrm{w}}\|_{p,\Sigma_T}+
\|\gamma_{\mathrm{e}}\|_{p,\Gamma\times ]0,T[})+\\
 + \sqrt{k_\#(1+T\exp[T])}
\left(\sqrt{\frac{2(\ell-1)}{\ell (b_\#)^{1/(\ell-1)}}}\|\gamma_{\mathrm{w}}
\|_{\ell\,'
,\Sigma_T} ^{\ell\,'/2}
+\right. \\ \left. +\sqrt{2+k_\#}
K_{2n/(n+1)}|\Omega|^{1/(2n)}
\|\gamma_{\mathrm{e}}
\|_{2,\Gamma\times ]0,T[} \right) .
\end{align*}
\end{proposition}
\begin{proof}
The existence of the required auxiliary temperature is consequence of
Theorem \ref{tmain} and Remark \ref{ral}, by dividing (\ref{pbta}) by $\rho  c_\mathrm{p}>0$.
\end{proof}

The continuous dependence is stated in the following proposition.
\begin{proposition}\label{corom}
 The mapping $\mathcal{T}$ is continuous and compact from
  $\mathcal{K} $ into 
\[
[L^p(0,T;W^{1,p}(\Omega))]^I\times V_{p,\ell}(Q_T)
\]
 for the strong topology.
\end{proposition}
 \begin{proof}
 Let $\{({\bf c}_m,\theta_m)\}_{m\in\mathbb{N}}\subset \mathcal{K}$
 be a sequence  such that
  \[ ({\bf c}_m,\theta_m)\rightarrow ({\bf c},\theta)\quad
\mbox{ in }[L^p(0,T;W^{1,p}(\Omega))]^I\times V_{p,\ell}(Q_T).
\]
 Clearly that
  $({\bf c},\theta)\in \mathcal{K}$. 
We select a weakly converging subsequence with respect to the norms from
 the estimates (\ref{cotaf2})-(\ref{cotaep}),
(\ref{cotapsii}) and (\ref{cotata}).
That is,   the corresponding solutions
 $(\phi_m,{\bf \Psi}_m,\Theta_m)$  
 in accordance with  Propositions \ref{lemaep}, \ref{lemaci}, and \ref{lemata}
verify
  $\phi_m\rightharpoonup\phi$ in $W^{1,p}(\Omega)/\mathbb{R}$, and
  $({\bf \Psi}_m,\Theta_m)\rightharpoonup({\bf \Psi},\Theta)$
 in $[ L^p(0,T;W^{1,p}(\Omega))]^{I+1}$.
 Moreover,   $\phi_m\rightharpoonup\phi$
 in $V_{p}(\Omega)$. 
  Under the compact embeddings  $W^{1,p}(\Omega)\hookrightarrow\hookrightarrow
 L^{p}(\Omega) $ and  $W^{1,p}(\Omega)\hookrightarrow$
$  \hookrightarrow L^{p}(\partial\Omega) $
  the compactness  Aubin-Lions theorem states that
  we may extract a sequence in the set of approximate  concentrations
and temperature
solutions, $({\bf \Psi}_m,\Theta_m)$, 
which converges strongly in $L^p(Q_T)$ and in $L^p(\Sigma_T)$.
Thanks to  (\ref{cotatlp}), $\Theta_m\rightarrow\Theta $ in $L^\ell(\Sigma_T)$.

The above limits ensure that the weak limit $(\Phi,{\bf \Psi},\Theta)$
verifies  $(\Phi,{\bf \Psi},\Theta)=\mathcal{T}({\bf c},\theta)$.

Next we prove the strong convergence of $\phi_m$ to $\phi$.
Since the weak limit $\phi$ verifies (\ref{pbfi})
we write
\begin{align*}
\int_\Omega\sigma(\theta_m)\nabla(\phi_m-\phi)\cdot \nabla v\mathrm{dx}=
\int_\Omega(\sigma(\theta)-\sigma(\theta_m
))\nabla\phi\cdot \nabla v\mathrm{dx}\\
+\int_\Omega
\left(\alpha(\theta) \sigma(\theta) \nabla\theta-
\alpha(\theta_m) \sigma(\theta_m) \nabla\theta_m\right)
\cdot \nabla v\mathrm{dx}+\\
+
F\sum_{i=1}^Iz_i\int_\Omega
\left(D_i(\theta)\nabla c_i-D_i(\theta_m)\nabla (c_i)_m\right)
\cdot \nabla v\mathrm{dx}.
\end{align*}
Thus, we may estimate   $\nabla(\phi_m-\phi)$ in $L^p(\Omega)$ such that
  $\|\nabla(\phi_m-\phi)\|_{p,\Omega}\rightarrow 0$ as $m$ tends to infinity.
 
Finally the strong convergence for the concentrations-temperature pair
is obtained via the identities
 \begin{align*}
\int_0^T\langle \partial_t \left((\Psi_i)_m-\Psi_i\right)
, v\rangle\mathrm{dt}+
\int_{Q_T}D_i(\theta_m)\nabla  \left((\Psi_i)_m-\Psi_i\right)
\cdot \nabla v\mathrm{dx}\mathrm{dt} =\\
=\int_{Q_T}  \left(D_i(\theta)-D_i(\theta_m)\right)\nabla\Psi_i
\cdot \nabla v\mathrm{dx}\mathrm{dt} +\\
+\int_{Q_T}\left( c_iS_i(c_i,\theta)\nabla\theta-(c_i)_mS_i(
(c_i)_m,\theta_m)\nabla\theta_m
\right)\cdot \nabla v\mathrm{dx}\mathrm{dt}+\\+ 
\int_{Q_T}\frac{t_i}{Fz_i}\left(\sigma(\theta)\nabla\phi-
\sigma(\theta_m)\nabla\phi_m\right)
\cdot \nabla v\mathrm{dx}\mathrm{dt}+\\
+\int_0^T\int_{\Gamma}\left(g_i(\theta_m,\phi_m)-g_i(\theta,\phi)\right)
v\mathrm{ds}\mathrm{dt}
,\quad \forall  v\in L^{p'}(0,T;W^{1,p'}(\Omega));
\\ 
\rho c_\mathrm{p}
\int_0^T\langle \partial_t \left(\Theta_m-\Theta\right)
, v\rangle\mathrm{dt}+
\int_{Q_T}\left(\mathsf{K}(\theta_m)\nabla (\Theta_m-\Theta)\right)
\cdot \nabla v\mathrm{dx}
\mathrm{dt}
=\\=
\int_{Q_T}\left((\mathsf{K}(\theta)-\mathsf{K}(\theta_m))
\nabla\Theta\right)\cdot \nabla v\mathrm{dx}
\mathrm{dt}+\\
+\int_{\Sigma_T}\left( h_\mathrm{R} (\theta)|\Theta|^{\ell-2}
\Theta- h_\mathrm{R} (\theta_m)|\Theta_m|^{\ell-2}
\Theta_m+
\gamma(\theta_m) -\gamma(\theta) \right) v\mathrm{ds}\mathrm{dt}+  \\
+\int_0^T\int_{\Gamma}\left( h_\mathrm{C} (\theta)\Theta-
h_\mathrm{C} (\theta_m)\Theta_m+\left(h_\mathrm{C}(\theta_m)-h_\mathrm{C}(\theta)\right)
\theta_\mathrm{e}\right) v
\mathrm{ds}\mathrm{dt}+\\
 +R\int_{Q_T}\sum_{j=1}^I \left(\theta^2D_j'(c_j,\theta)\nabla c_j-
\theta_m^2D_j'((c_j)_m,\theta_m)\nabla (c_j)_m\right) 
\cdot\nabla v\mathrm{dx}\mathrm{dt}+\\
 +\int_{Q_T} \left(
\sigma(\theta) \Pi (\theta)\nabla\phi-
\sigma(\theta_m) \Pi (\theta_m)\nabla\phi_m\right) 
\cdot\nabla v\mathrm{dx}\mathrm{dt}
,\quad\forall v\in V_{p',\ell}(Q_T).
 \end{align*} 
 
Indeed, the estimates 
(\ref{cotapsii}) and (\ref{cotata}) applied to the differences
$(\Psi_i)_m-\Psi_i$ and $\Theta_m-\Theta$, respectively, yield
  their convergence to zero by the Lebesgue dominated convergence theorem.
\end{proof}
 
  \section{Proof of Theorem \ref{main}}
 \label{proof}

The functional $\mathcal T$ (cf. (\ref{deft})) is  well defined from
  $\mathcal{K} $ into $[L^p(0,T;W^{1,p}(\Omega))]^I
\times V_{p,\ell}(Q_T)$ by
 Propositions \ref{lemaep}, \ref{lemaci}, and \ref{lemata}.
 Its continuity is ensured by Proposition \ref{corom}.
 In order to apply the Schauder fixed point theorem it remains to prove that
  $\mathcal T$ maps  $\mathcal K$ into itself.
  To this aim, let   $({\bf c},\theta)\in \mathcal{K}$ be arbitrary
in order to show that 
  $\mathcal{T}({\bf c},\theta)\in \mathcal{K}$.
 First,  we rewrite (\ref{cotaf2})-(\ref{cotaep}) as
\begin{equation}
\| \nabla \phi\|_{p,Q_T}\leq B^\#+A^\#\left(\sigma^\#\alpha^\# R
+\sum_{j=1}^I  D_{j}^\#R_j\right),\label{cotaf}
\end{equation}
 with
 \begin{align}
 A^\#= (\sigma_\#)^{-1} \left(M_1
|\Omega|^{1/2-1/p}
+M_2\sqrt{ 1+\sigma_\# }\right);\label{defas}\\
B^\#= (\sigma_\#)^{-1} T^{1/p}\left(
M_1K \|g\|_{2 ,\Gamma}
 +M_3\sqrt{2+{ 2^{-1/n}\sigma_\#}}
 \| g\|_{p,\Gamma}\right).\label{defbs}
\end{align}

Secondly, we assume that
\begin{align*}
K_{2n/( n+1)}^{2/p}|\Omega|^{(pn)^{-1}}
\left[
 \|\gamma_i\|_{p',\Sigma_T} +g_i^\#K_{pn/(n+p-1)}|\Omega|^{1-1/p}\times
 \right.\\
\left.\times
\left(  B^\#+\left(1+P_pA^\#\sigma^\#\alpha^\#\right) R
+P_pA^\#\sum_{j=1}^I  D_{j}^\#R_j\right)\right]>1,
\end{align*}
otherwise an easier argument can be applied. Thus,
we insert (\ref{cotaf}) into 
(\ref{cotapsiip})-(\ref{cotapsii}) resulting in
\begin{align}
\|\Psi_i\|_{p,Q_T} +\|\nabla \Psi_i\|_{p,Q_T} \leq \mathcal{A}_i^0R+
\mathcal{A}_i \sum_{j=1}^I (D_j')^\#R_j+\nonumber \\+
\left(T\exp\left[ (p-1)T\right]\right)^{1/p}\left[
\|c_{0,i}\|_{p,\Omega}+Q_i^\#\|\gamma_i\|_{p',\Sigma_T} +\right.\nonumber\\
\left. +
\left(
\sqrt{\frac{p-1}{(D_i)_\#}}t_i^\#\sigma^\#
 +g_i^\# K_{pn/( n+p-1)}|\Omega|^{1-1/p}
P_p\right)B^\#\right]+\nonumber\\
+ \mathcal{C} (D_i)_\#^{-1}
 \left[ \sqrt{(D_i)_\#(1+T\exp[T])}\|c_{0,i}\|_{2,\Omega}  +\mathcal{G}_i^\#
+Y_iB^\#+
\right.\nonumber\\ 
\left. 
+X_i(\sigma_\#)^{-1}T^{1/2}K \|g\|_{2,\Gamma}
\right],\label{psipp}
 \end{align}
with
\begin{align}
\mathcal{A}_i^0=
\left(T\exp\left[ (p-1)T\right]\right)^{1/p}\left[
\sqrt{\frac{p-1}{(D_i)_\#}}\left(S_i^\#+
A^\#t_i^\#(\sigma^\#)^2\alpha^\# \right)+\right.\nonumber\\
\left.
 +g_i^\# Q_i^\# K_{pn/(n+p-1)}|\Omega|^{1-1/p}
\left(1+P_pA^\#\sigma^\#\alpha^\# \right)\right]+\nonumber\\
+\mathcal{C} (D_i)_\#^{-1}
 \left[ S_i^\#
\mathcal{Z}(|Q_T|^{1/2-1/p},(D_i)_\#,1)
+\mathcal{Q}_i+\right.\nonumber \\
\left.+(X_i(\sigma_\#)^{-1}|Q_T|^{1/2-1/p}+Y_iA^\#)
\sigma^\#\alpha^\# \right];\label{ai0}
\end{align}\begin{align}
\mathcal{A}_i=\mathcal{C} (D_i)_\#^{-1}
 \left( X_i(\sigma_\#)^{-1}|Q_T|^{1/2-1/p}+Y_i A^\#\right)
+\nonumber\\
+A^\#
\left(T\exp\left[ (p-1)T\right]\right)^{\frac{1}{p}}
\left(\sqrt{\frac{p-1}{(D_i)_\#}} t_i^\#\sigma^\#+
g_i^\# 
Q_i^\# K_{pn/( n+p-1)}|\Omega|^{1-\frac1p}
P_p\right)  ;\label{ai}\\
Q_i^\#=\left( \left(\frac{p^2(p-1)^{p-2}}{2(D_i)_\#}\right)^{1/( p-1)}
+p-1\right)^{1/ p}
K_{2n/(n+1)}^{2/p}|\Omega|^{(pn)^{-1}},\nonumber
\end{align}
where $\mathcal{G}_i^\#$, $X_i$, $Y_i$, $\mathcal{Q}_i$,
 $\mathcal{Z}$, and  $A^\#$, $B^\#$
 are given at (\ref{gii})-(\ref{defz}), and (\ref{defas})-(\ref{defbs}),
 respectively.

Next,  by the one hand, 
we insert (\ref{cotaf2}) into  (\ref{cotatlp}) resulting in
\begin{align}
\|\Theta\|_{\ell,\Sigma_T}^\ell\leq
\frac{1+T \exp\left[T\right]}{b_\#}
\left( \rho c_\mathrm{p}\|\theta_0\|_{2,\Omega}^2
+\frac{2(\ell-1)}{\ell (b_\#)^{1/(\ell-1)}}\|\gamma_{\mathrm{w}}\|_{\ell\,'
,\Sigma_T} ^{\ell\,'}+\right. \nonumber \\ 
+\left(\frac{2}{k_\#}+ \frac{1}{\rho c_\mathrm{p}} \right)K_{2n/(n+1)}^{2}|\Omega|^{n^{-1}}
\|\gamma_{\mathrm{e}}\|_{2
,\Gamma\times ]0,T[} ^2
+
\frac{1}{ k_\#}\left[
\frac{\Pi^\#\sigma^\#}{\sigma_\#}  
\sqrt{T} K \|g\|_{2,\Gamma}
\right. +\nonumber\\ +\left.\left.
 |Q_T|^{1/2-1/p}\left( \frac{\Pi^\#\alpha^\#(\sigma^\#)^2}{\sigma_\#}R+
 \left(1+\frac{\Pi^\#\sigma^\#}{\sigma_\#} \right)\sum_{j=1}^I (D_j')^\#R_j\right)\right]^2\right).\nonumber 
\end{align}
Since $\ell\geq 2$, we assume that
\begin{align*}
\frac{\Pi^\#\sigma^\#}{\sigma_\#} \left(\sqrt{T}K \|g\|_{2,\Gamma}
+  |Q_T|^{1/2-1/p}\alpha^\# \sigma^\#R\right)+\\
+ |Q_T|^{1/2-1/p}\left(1+\frac{\Pi^\#\sigma^\#}{\sigma_\#} \right)\sum_{j=1}^I (D_j')^\#R_j
>1,
\end{align*}
otherwise this term is upper bounded by one, and an easier argument can be applied. Thus, using the above 
inequalities, and inserting   (\ref{cotaf}) into (\ref{cotata}) we find
\begin{align}
\|\Theta\|_{\ell,\Sigma_T}+
\|\nabla \Theta\|_{p,\Omega} \leq\mathcal{B}_0R+
\mathcal{B} \sum_{j=1}^I (D_j')^\#R_j+\nonumber \\+
 \left(
\frac{1+T \exp\left[T\right]}{b_\#}\right)^{1/\ell}
\left[\left( \rho c_\mathrm{p}\|\theta_0\|_{2,\Omega}^2
+\frac{2(\ell-1)}{\ell (b_\#)^{1/(\ell-1)}}\|\gamma_{\mathrm{w}}\|_{\ell\,'
,\Sigma_T} ^{\ell\,'}+\right.\right. \nonumber \\ \left.\left.
+\left(\frac{2}{k_\#}+\frac{1}{\rho c_\mathrm{p}} \right)K_{2n/(n+1)}^{2}|\Omega|^{n^{-1}}
\|\gamma_{\mathrm{e}}\|_{2
,\Gamma\times ]0,T[} ^2\right)^{1/\ell}  +
\frac{\Pi^\#\sigma^\#}{(k_\#)^{1/\ell}\sigma_\#} 
T^{1-1/p} K \|g\|_{2,\Gamma}
\right]   +\nonumber\\
+\mathcal{C} (k_\#)^{-1}
 \left[ 
\sqrt{\rho c_\mathrm{p}k_\#(1+T\exp[T])}\|\theta_{0}\|_{2,\Omega}+\mathcal{H}^\#+
\right.\nonumber\\ 
\left.+\Pi^\#\sigma^\#
\mathcal{Z}(
T^{1/2}K \|g\|_{2,\Gamma}(\sigma_\#)^{-1}, (\rho c_\mathrm{p})^{-1}k_\#, B^\#)
\right],\label{ttlp}
 \end{align}
 where 
\begin{align} 
\mathcal{B}_0=
\frac{ \Pi^\#\alpha^\# (\sigma^\#)^{2}} {\sigma_\#} \left(
\frac{ \mathcal{C}\sqrt{ 1 +(\rho c_\mathrm{p})^{-1} k_\#} }{k_\# } 
\sigma_\#  A^\# +\right.\nonumber \\ \left.+
 \left[
\frac{ \mathcal{C}\sqrt{1+T \exp\left[T\right] }}{k_\#} +\left(
\frac{1+T \exp\left[T\right] }{b_\#k_\#}\right)^{1/\ell} \right] |Q_T|^{1/2-1/p}  \right) ; \label{defbb0} \\
\mathcal{B}=
\frac{ \mathcal{C}\sqrt{ 1 +(\rho c_\mathrm{p})^{-1} k_\#} }{k_\# } (1+\Pi^\#\sigma^\# A^\#)+\nonumber\\
 +
 \left[
\frac{ \mathcal{C}\sqrt{1+T \exp\left[T\right] }}{k_\#} +\left(
\frac{1+T \exp\left[T\right] }{b_\#k_\#}\right)^{1/\ell} \right] 
 \left(1+\frac{\Pi^\#\sigma^\#}{\sigma_\#} \right)|Q_T|^{\frac12-\frac1p}
. \label{defbb}
\end{align}

We seek for $(R,R_1,\cdots,R_I)$ such that $({\bf \Psi},\Theta)\in\mathcal{K}$.
According to (\ref{ttlp}), we define the continuous function
\[
\mathcal{P}(r)=
\left(1-\mathcal{B}_0 \right)r
-\mathcal{P}(0),
\]
where 
\[
\mathcal{P}(0)=C+ \mathcal{B}
\sum_{j=1}^I (D_j')^\#R_j>0,
\]
with the constant $C>0$ being independent on  $R,R_1,\cdots,R_I$.

  For our purposes in the finding of the
explicit smallness conditions
on the data, we choose $
R=\mathcal{P}(0)/(1-\mathcal{B}_0)$ as its positive root, 
 considering
the first smallness condition 
 \begin{align}
 \mathcal{B}_0<1.\label{datar}
\end{align}
With this choice we may define in a recurrence manner the following linear functions,
in accordance with (\ref{psipp}),
\begin{align*}
\mathcal{P}_1(r)=-\mathcal{P}_1(0)+\left(1-\mathcal{B}_1(D_1')^\# \right)r;\\
\mathcal{P}_2(r)=-\mathcal{P}_2(0)+
\left(1-\mathcal{B}_2(D_2')^\# (1-\frac{\mathcal{B}_1}{
1-\mathcal{B}_1(D_1')^\#})\right)r;\\
\mathcal{P}_3(r)= -\mathcal{P}_3(0)+\\
+\left(1-\mathcal{B}_3(D_3')^\# (1-\frac{\mathcal{B}_1}{
1-\mathcal{B}_1(D_1')^\#}-
\frac{\mathcal{B}_2}{1-\mathcal{B}_2(D_2')^\# (1-\frac{\mathcal{B}_1}{
1-\mathcal{B}_1(D_1')^\#})})
\right)r ,
\end{align*}
where
\[ 
\mathcal{B}_i:=\frac{
\mathcal{A}_i^0\mathcal{B} }{1-\mathcal{B}_0}+\mathcal{A}_i,
\]
where $\mathcal{B}_0$, $\mathcal{B}$, $\mathcal{A}_i^0$, and $\mathcal{A}_i$ are given at 
(\ref{defbb0}), (\ref{defbb}), (\ref{ai0}),
and (\ref{ai}), respectively.
All functions admit positive roots
(we call them $R_1,\cdots,R_I$) since $\mathcal{P}_i(0)>0$ for $i=1,\cdots,I$,
and the smallness conditions $\mathcal{P}_i'(r)>0$ {\em i.e.}
\begin{align}
\mathcal{B}_1(D_1')^\#<1 ; &&\label{datar1}\\
\mathcal{B}_{i}(D_{i}')^\#\left(1-\sum_{j=1}^{i-1}
\frac{\mathcal{B}_j}{ \mathcal{P}_j'(r)}\right)<1 ,&&
i=2,\cdots,I, \label{datari}
 \end{align}
hold.  For reader's convenience, we rewrite the above smallness conditions to the first two ionic components
 \begin{align*} 
 \mathcal{B}_1: = \frac{
\mathcal{A}_1^0\mathcal{B} }{1-\mathcal{B}_0}+\mathcal{A}_1 < ((D_1')^\# )^{-1} ;\\
 \mathcal{B}_2 := \frac{
\mathcal{A}_2^0\mathcal{B} }{1-\mathcal{B}_0}+\mathcal{A}_2 <
((D_{2}')^\# )^{-1} \frac{1 - \mathcal{B}_1 (D_1')^\# }{ 1- \mathcal{B}_1 (1+(D_1')^\# ) }.
\end{align*}

\section{Electrolysis of molten sodium chloride}
\label{nacl}

Many metals can be extracted in pure forms by electrolytic method: the alkali metals, 
and aluminum, as well as nonmetals: oxygen, hydrogen, and  chlorine gas.
We exemplify the electrolytic cell (cf. Fig. 1) for NaCl, with $\rho=$ 1500 kg$ \cdot $m$^{-3}$ and 
$c_\mathrm{p}=1197.8$ J$ \cdot $kg$^{-1} \cdot $K$^{-1}$. As in the 
industrial extraction of the sodium metal by Downs process,
we consider
  a cylindrical container  (with dimensions of 13 cm in diameter, and of 13 cm in  height) with stainless steel walls
($\ell=5$, the emissivity $0.2\leq \epsilon\leq 0.5$, and the absorptivity is assumed to obey the Kirchhoff law), 
and with copper/nickel electrodes  ($550< h_\mathrm{C}\leq 1820$ W$\cdot $m$^{-2}\cdot $K$^{-1}$ \cite{utig}).
Thus, we suppose $|\Omega |=1.5\times 10^{-3}$ m$^3$, which corresponds to
 $c_i^0=2.5667\times 10^4$ mol$\cdot $m$^{-3}$ ($i=$ Na$^+$, Cl$^-$).

The sodium chloride conducts electricity when it is melted (high melting point 1073.15 K).
At temperature range 1080 -- 1250 K (805 -- 980$^\circ$C), we have the following available data: 
 $k^\#=0.6$ and $k_\#=0.5$ W$\cdot $m$^{-1}\cdot $K$^{-1}$ \cite{galamba}, 
$\sigma_\#=359.7 $ S$ \cdot $m$^{-1}$,
$\sigma^\#=398.0 $ S$ \cdot $m$^{-1}$,
$(D_\mathrm{Na^+})_\#=7.7\times 10^{-9}$,
$(D_\mathrm{Cl^-})_\#=6.3\times 10^{-9}$ m$^2 \cdot $s$^{-1}$,
$D_\mathrm{Na^+}^\# =12\times 10^{-9} F|z_\mathrm{Na^+}|$,
$D_\mathrm{Cl^-}^\# =9.5\times 10^{-9} F|z_\mathrm{Cl^-} | $ m$^2 \cdot $s$^{-1}\cdot $C$ \cdot $mol$^{-1}$
\cite[pp. 49-63]{janz}. The Seebeck coefficient has values in the range
 $   10^{-5}-10^{-4}$ V$ \cdot $K$^{-1}$ \cite{wurger}.
 The  parameters,  $\Pi^\#$, and $(D_i')^\#$ ($i=$ Na$^+$, Cl$^-$), 
 are according  to, respectively, 
  the first Kelvin relation, and the Onsager reciprocal relationship. 

 Under constant initial conditions, 
the upper bound in (\ref{tmm}) can be given by
$ t_i^\#= F|z_i|D_i^\# c_i^0/(R\theta_0\sigma_\#)$.
 The Soret coefficient (S/D) is of order $10^{-3}$ -- $10^{-2}$ K$^{-1}$
 in liquids and electrolytes \cite{platten}, which implies $S_\mathrm{Na^+}^\# =1.2\times 10^{-12} c^0_\mathrm{Na^+} $
 and $S_\mathrm{Cl^-}^\# = 9.5\times 10^{-11} c^0_\mathrm{Cl^-}$.

The electrolysis separates the molten ionic compound into its elements.
The chemical half-reactions (and the standard  state potentials)  are:
\begin{itemize}
\item in the cathode (-): $2\mathrm{Na}^++2e^- \longrightarrow 2\mathrm{Na}$ ($E^0_\mathrm{redution}=-2.71$ V);
\item in the anode (+): $2\mathrm{Cl}^- \longrightarrow\mathrm{Cl}_2(g) +2e^-$ ($E^0_\mathrm{oxidation}=-1.36$ V).
\end{itemize}
Thus, the balanced chemical equation for the nonspontaneous  overall reaction is
\[ 2 \mathrm{NaCl} \longrightarrow 2\mathrm{Na} + \mathrm{Cl}_2(g)\qquad (E^0_\mathrm{cell}=-4.07 \mathrm{V}) . \]
The stoichiometric coefficients of electrons in the anode and cathode are,
respectively, $s_\mathrm{a}=s_\mathrm{c}=2$.
Assuming symmetric electron transfer, the transfer coefficients are $\beta_i=0.5$ ($i=$ Na$^+$, Cl$^-$).
Then, the Butler-Volmer equation is
$g_{i,l}= 2
J_{l} \sinh\left[F\eta/(R\theta)\right]$.

The production of metallic sodium at the cathode and chloride gas at the anode may operate at 
 $10^{4}$ A$ \cdot $m$^{-2}$, and at potential of 7 V, with
the cathodic current being balanced by the anodic current.

Therefore, for some $T>0$  the  smallness conditions (\ref{datar})-(\ref{datari}) hold
under the above data, and 
\begin{align*}
\mathcal{B}_0=& 0.0027
\left( 2 \mathcal{C}(M_1+ 18.99M_2+ \sqrt{1+T\exp[T]} )
+
44.643\left(1+T \exp\left[T\right] \right)^{1/5}\right);\\
\mathcal{B}= &
  48.9  (1+T\exp[T])^{1/5} + \mathcal{C}
\left[ 2  \sqrt{1+T\exp[T]} +2 
\right] ;\end{align*}\begin{align*}
\mathcal{A}^0_{Na^+}= &(T\exp[(p-1)T])^{1/p}\left[ 
 0.035+ 0.0032M_1 + 0.061 M_2
\right]  +\\ &
 +\mathcal{C}\left[ 400\sqrt{1+T\exp[T]} + 436.8
+ 36.8M_1 + 699.6 M_2
\right];\\
\mathcal{A}_{Na^+}=& \mathcal{C} ( 1322.2 
+   1322.2 M_1+ 25111.5 M_2)+ \\ & +
( 0.116 M_1 + 2.2 M_2) (T\exp[(p-1)T])^{1/p}  ; \\ & \qquad\qquad
 ((D_{Na^+}')^\# )^{-1} = 6.9281 \times 10^5.
\end{align*}
Since the values of parameters for Cl$^{-}$ are of the same order of
 the ones for Na$^{+}$, then $\mathcal{A}_{Cl^{-}}^0$ and 
 $\mathcal{A}_{Cl^{-}}$ have similar expressions.
Further optimization work should be done to precise the above universal constants.
Their quantitative form is being a matter of study of ongoing work.

\begin{table}[h]
\centering  \caption{Universal constants} \label{tab}\small
\begin{tabular}{|c|c|c|}
\hline
 $F$  & Faraday constant &$9.6485 \times 10^4 $ C$\cdot $mol$^{-1}$\\
\hline
 $R$ & gas constant & $8.314 $ J$\cdot $mol$^{-1}\cdot $K$^{-1}$\\
\hline
$\sigma_{\rm SB}$& Stefan-Boltzmann constant
& $5.67\times 10^{-8}$ W$\cdot $m$^{-2} \cdot $K$^{-4}$\\
& (for blackbodies)&\\
\hline
\end{tabular}
\end{table}

\section*{Appendix}

{\bf Nomenclature list:}

$
\begin{array}{lll}
c&\mbox{molar concentration}&\mbox{mol$ \cdot $m}^{-3}\\
 c_\mathrm{p} &\mbox{specific heat capacity}&
\mbox{J$\cdot $kg$^{-1}\cdot $K}^{-1} \\
 D &\mbox{diffusion coefficient} &\mbox{m$^2\cdot $s$^{-1}$}\\
 D' &\mbox{Dufour coefficient}&\mbox{m$^2\cdot $s$^{-1}\cdot $K}^{-1}\\
 h &\mbox{heat transfer coefficient}&\mbox{W$\cdot $m$^{-2}\cdot $K}^{-1}\\
 k &\mbox{thermal conductivity}&\mbox{W$\cdot $m$^{-1}\cdot $K}^{-1}\\
 S &\mbox{Soret coefficient (thermal diffusion)}&\mbox{m$^2\cdot $s$^{-1}\cdot $K}^{-1}\\
 t& \mbox{transference number} &\mbox{(dimensionless)}\\
 u &\mbox{mobility}&\mbox{m$^2\cdot $V$^{-1}\cdot $s}^{-1}\\
 z &\mbox{valence} &\mbox{(dimensionless)}\\
\alpha  &\mbox{Seebeck coefficient} &\mbox{V$ \cdot $K}^{-1}\\
\phi &\mbox{electric potential}& \mbox{V}\\
 \Pi &\mbox{Peltier coefficient} &\mbox{V}\\
\rho &\mbox{density} &\mbox{kg$\cdot $m}^{-3}\\
\sigma &\mbox{electrical conductivity}&\mbox{S$ \cdot $m}^{-1}\\
\theta &\mbox{absolute temperature} &\mbox{K}
\end{array}
$

\subsection*{Acknowledgment}
The author would like to thank
the reviewer for developing the original manuscript.
The final publication is
available at link.springer.com.

\end{document}